\documentclass[]{gAPA2e}

\newcommand{\wA}{\widetilde A}

\newcommand{\supp}{\operatorname{supp}}

\newcommand{\pr}{\operatorname{pr}}
\newcommand{\inj}{\operatorname{i}}

\newcommand{\comega}{\overline\Omega }

\newcommand{\ang}[1]{\langle {#1} \rangle}
\newcommand{\Op}{\operatorname{Op}}

\newcommand{\simto}{\overset\sim\rightarrow}

\newcommand{\Ami}{A_{\min}}
\newcommand{\Ama}{A_{\max}}

\begin{document}
\doi{10.1080/0003681YYxxxxxxxx}
 \issn{1563-504X}
\issnp{0003-6811}
\jvol{00} \jnum{00} \jyear{2008} \jmonth{January}

\markboth{G. Grubb}{Peturbation of Essential Spectra}


\title{{\itshape Perturbation of essential spectra of exterior 
elliptic problems}}

\author{Gerd Grubb$^{\ast}$\thanks{$^\ast$Corresponding
    author. Email: {\tt grubb\@math.ku.dk}
\vspace{6pt}} \\\vspace{6pt}  {\em{Copenhagen University Department of
  Mathematical Sciences,
  Universitetsparken 5, DK-2100 Copenhagen, Denmark
}}\\\vspace{6pt}\received{received December  2008} }

\maketitle

\begin{abstract}
For a second-order symmetric strongly elliptic differential operator on an
exterior domain in ${\mathbb
R}^n$ it is known from works of Birman and Solomiak
that a
change of the boundary condition from the Dirichlet condition to an
elliptic Neumann or Robin condition leaves the essential spectrum
unchanged, in such a way that the spectrum of the
difference between the inverses satisfies a Weyl-type asymptotic formula. 
We show that one can increase, but
not diminish, the essential spectrum by imposition of other
Neumann-type non-elliptic boundary conditions. 

The results are
extended to  $2m$-order operators, where it is shown that for any
selfadjoint realization defined by an elliptic normal boundary
condition (other than the Dirichlet condition), one can augment the 
essential spectrum at will by adding a suitable
operator to the mapping from free Dirichlet data to Neumann data.
We here also show an extension of the
spectral asymptotics formula for the difference between inverses of
elliptic problems. 

The proofs rely on Kre\u\i{}n-type formulas for
differences between inverses, and cutoff techniques, combined with 
results on singular Green operators and their spectral asymptotics.

\medskip

{\it Dedicated to Vsevolod A.\ Solonnikov on the occasion of
his 75 years birthday.}
\bigskip

\begin{classcode}35J40, 35P20, 35S15, 47A10 \end{classcode}

\bigskip

\end{abstract}

\section{Introduction}\label{Introduction}

Let $A$ be  a uniformly strongly elliptic differential
operator on ${\mathbb R}^n$ ($n\ge 2$)
\begin{equation}
A=-\sum_{j,k=1}^n\partial_ja_{jk}(x)\partial_k+a_0(x),\label{tag1.1}
\end{equation} 
with real bounded smooth coefficients with bounded derivatives,
satisfying $a_{jk}=a_{kj}$ and
\begin{equation}
 \sum_{j,k}a_{jk}(x)\xi _j\xi _k\ge c_1|\xi |^2,\; a_0(x)\ge c_2,\text{
for }x,\xi \in{\mathbb R}^n,\label{tag1.2}
\end{equation}
with $c_1,c_2>0$. We denote by $A_0$ 
the maximal realization in $L_2({\mathbb R}^n)$; it is selfadjoint
positive. Let $\Omega _+\subset{\mathbb R}^n$ be the exterior of a bounded smooth open
set $\Omega _-$, with boundary denoted $\Sigma $ ($=\partial\Omega _+=\partial\Omega _-$),
and let
$A_1$, $A_2$ and $A_3$ be the selfadjoint lower bounded realizations
in $L_2(\Omega _+)$
determined by
the Dirichlet condition ($\gamma _0u\equiv u|_{\Sigma  }=0$), the oblique Neumann
condition ($\nu _{A}u=0$, see \eqref{tag2.4} below),  resp.\  a Robin 
condition ($\nu _{A}u=b(x)\gamma _0u$) with $b$ real and smooth. The coefficient $a_0$ is assumed to be taken so large positive that all four operators 
have positive lower bound. 

It is known that the operators $A_j$ have an
unbounded essential spectrum, consisting of an interval $[c,\infty [\,$ if the
coefficients converge to a limit for $|x|\to\infty $, and more
generally being a 
subset of $[c,\infty [\,$ with possible gaps (e.g.\ when
the coefficients are periodic).

Birman showed in \cite{B62} a general principle concerning the
stability of the essential spectrum:
\begin{align}
A_0^{-1}&-A_j^{-1}\oplus 0_{L_2(\Omega _-)}\in T_{2/n},\label{tag1.3} \\
A_j^{-1}&-A_k^{-1}\in T_{2/(n-1)}, \text{ for }j,k=1,2,3;
\label{tag1.4}
\end{align}
where $T_\alpha $ denotes the class of compact operators whose
characteristic values $s _l$ are $O(l^{-\alpha })$ for $l\to\infty
$. (When $\Omega _1\cup\Omega _2$ is a disjoint union of open sets,
and $P_i$ acts in $L_2(\Omega _i)$, we denote by $P_1\oplus P_2$ the
operator in $L_2(\Omega _1\cup\Omega _2)$ that acts like $P_i$ on
$L_2(\Omega _i)$, naturally injected in $L_2(\Omega _1\cup\Omega _2)$.) In
particular, all four operators have the same essential spectrum $\sigma _{\operatorname{ess}}(A_0)$;
this extends a result of Povzner, as referred to in \cite{B62}.
(Birman's paper also allowed unbounded coefficients and limited smoothness, but we shall not
follow up on those aspects here.)

The result \eqref{tag1.4} was refined by Birman and Solomiak in \cite{BS80}, where a
Weyl-type spectral asymptotics formula was obtained ($s_l l^{2/(n-1)}
$ converges
to a limit for $l\to\infty $). 
In Grubb \cite{G84} similar spectral asymptotics formulas were shown
by methods of pseudodifferential boundary
problems, 
and refinements with a spectral resolvent parameter were 
studied in \cite{G84a}. 
Spectral estimates of resolvent differences have been taken up again
in recent works of Alpay and Behrndt \cite{AB09},  
Gesztesy and Malamud \cite{GM08}.

The present paper extends the results to higher-order operators, but
aims in particular for a slightly different question, 
namely of how much one
can perturb the essential spectrum of $A_3$ by replacing the Robin
condition by a more
general {\it Neumann-type} boundary condition (not
necessarily elliptic)
\begin{equation}
\nu _{A}u=C\gamma _0u. \label{tag1.5}
\end{equation}
Let $\wA$ denote the realization of $A$ on $\Omega _+$ determined by
\eqref{tag1.5}, i.e., with domain
\begin{equation}
D(\wA)=\{u\in L_2(\Omega _+)\mid Au\in L_2(\Omega _+),\,\nu _{A}u=C\gamma _0u \}.
\label{tag1.6}\end{equation}
The outcome is as follows:

1) For any nonzero $a\in {\mathbb R}\setminus \sigma _{\operatorname{ess}}(A_0)$, $C$ can
be chosen as a pseudodifferential operator of order $1$ such that
$\wA$ is selfadjoint with \begin{equation}
\sigma _{\operatorname{ess}}(\wA)=\sigma
_{\operatorname{ess}}(A_0)\cup \{a\}.\label{tag1.7}
\end{equation} More generally, when $T_0$ is an invertible selfadjoint
operator in a separable infinite dimensional Hilbert space $Z_0$, one 
can choose an operator $C$ such 
that $\wA$ is selfadjoint and
\begin{equation}\sigma _{\operatorname{ess}}(\wA)=\sigma _{\operatorname{ess}}(A_0)\cup\sigma _{\operatorname{ess}}(T_0).\label{tag1.8}\end{equation}

2) For any choice of an operator $C$ in \eqref{tag1.5} defining a selfadjoint invertible
realization $\wA$, $\sigma _{\operatorname{ess}}(A_0)$ remains in the essential spectrum of $\wA$.

We also reprove the spectral asymptotics formulas, and extend the results
to strongly elliptic operators of order $2m$ for positive integer $m$. 

The question of whether points of $\sigma _{ \operatorname{ess}}(A_0)$
can be removed by a perturbation of the boundary condition was brought
up in a conversation with M.\ Marletta, M.\ Brown and I.\ Wood in
Cardiff in May 2008; the author thanks these colleagues for useful discussions. 

\bigskip

\section {Description of the operators in the second-order case}\label{Description}

Let us first recall some well-known facts. The Sobolev space
$H^s({\mathbb R}^n)$ ($s\in {\mathbb R}$) can be provided with the norm
$\|u\|_s
=\|{\cal F}^{-1}(\ang\xi ^s{\cal
F}u)\|_{L_2({\mathbb R}^n)}$; here ${\cal F}$ is the Fourier transform and
$\ang\xi =(1+|\xi |^2)^{\frac12}$. There is a standard construction
from this of
Sobolev spaces over an open subset and over the
boundary manifold. We denote by $\Ama$ resp.\ $\Ami$ the
operators acting like $A$ with domains
\begin{equation*}
D(\Ama)=\{u\in L_2(\Omega _+)\mid Au\in L_2(\Omega _+)\}, \quad D(\Ami)=H^2_0(\Omega _+);
\end{equation*}
here $\Ami$ is closed symmetric, and $\Ama=\Ami^*$.  The operators $\wA$ satisfying $\Ami\subset \wA\subset \Ama$ are
called the realizations of $A$.

The symmetric sesquilinear forms
\begin{align}
s_{{\mathbb R}^n}(u,v)&=\int_{{\mathbb
R}^n}\sum_{j,k=1}^n(a_{jk}\partial_ku\partial_j\overline{v}+a_0u\overline
v)\,dx,\nonumber\\
 s(u,v)&=\int_{\Omega _+}\sum_{j,k=1}^n(a_{jk}\partial_ku\partial_j\overline{v}+a_0u\overline
v)\,dx, \label{tag2.1}
\end{align}
are bounded on $H^1({\mathbb R}^n)$ resp.\ $H^1(\Omega _+)$ and satisfy
\begin{equation}
s_{{\mathbb R}^n}(u,u)\ge c\|u\|_{H^1({\mathbb R}^n)}^2\text{ resp. }s(u,u)\ge
c\|u\|_{H^1(\Omega _+)}^2\label{tag2.2}
\end{equation}
there, with $c=\min\{c_1,c_2\}$. Moreover,
\begin{equation}
(Au,v)_{L_2(\Omega _+)}=s(u,v)+(\nu _Au,\gamma _0v)_{L_2(\Sigma )}, \quad u\in
H^2(\Omega _+), v\in H^1(\Omega _+),\label{tag2.3}
\end{equation}
where 
\begin{equation}
\nu _{A}u=\sum_{j,k}a_{jk}\nu _j\gamma _0\partial_k u,\label{tag2.4}
\end{equation}
with $(\nu _1(x),\dots,\nu _n(x))$ denoting the interior unit normal to
$\Omega _+$ at $x\in \Sigma $.
Hence the standard variational
construction (the Lax-Milgram lemma) applied to the triples \linebreak$(L_2({\mathbb R}^n),H^1({\mathbb
R}^n),
s_{{\mathbb R}^n}) $,  $(L_2(\Omega _+),H^1_0(\Omega _+),s) $,
resp.\ $(L_2(\Omega _+),H^1(\Omega _+),s) $ 
defines the positive selfadjoint operators $A_0$ in $L_2({\mathbb R}^n )$, $A_1$
and $A_2$ in $L_2(\Omega _+)$ mentioned in the introduction. (The
variational construction is known e.g.\ from Lions and Magenes
\cite{LM68}, and is also explained in Grubb \cite{G09}.) In view
of elliptic
regularity theory and the uniform symbol estimates, the
domains are in fact contained in $H^{2}$. Moreover, the operators
representing the 
nonhomogeneous boundary value problems (cf.\ e.g.\ \cite{LM68})
\begin{equation}
{\cal A}_1=\begin{pmatrix} A\\\gamma _0\end{pmatrix} : 
H^{s+2}(\Omega _+) \to
\begin{matrix}
H^{s}(\Omega _+)\\ \times \\ H^{s+\frac32}(\Sigma  )\end{matrix},
\quad
{\cal A}_2=\begin{pmatrix} A\\\nu  _A\end{pmatrix} : 
H^{s+2}(\Omega _+) \to
\begin{matrix}
H^{s}(\Omega _+)\\ \times \\ H^{s+\frac12}(\Sigma  )\end{matrix},
\label{tag2.5}
\end{equation}
where $s>-\frac32$ resp.\ $s>-\frac12$,
have solution operators, continuous in the opposite direction:
\begin{equation}
{\cal A}_1^{-1}=\begin{pmatrix} R_1&\; K_1\end{pmatrix},\quad 
{\cal A}_2^{-1}=\begin{pmatrix} R_2& \; K_2\end{pmatrix}.\label{tag2.6}
\end{equation}
In modern terminology,
\begin{equation}
R_1=Q_+-K_1\gamma _0Q_+,\quad R_2=Q_+-K_2\nu _AQ_+, \label{tag2.7}
\end{equation}
where $Q$ is the {\it pseudodifferential operator} $Q=A_0^{-1}$ on ${\mathbb
R}^n$ and $Q_\pm= r^\pm Qe^\pm $ is its truncation  
 to $\Omega _\pm$ (here $e^\pm $ extends to ${\mathbb R}^n$ by 0 on 
$\Omega _\mp$, $r^\pm $
restricts from ${\mathbb R}^n$ to $\Omega _\pm$). The operators  $K_1$ and $K_2$
are {\it Poisson operators} solving the respective boundary value problems
with nonzero boundary data, zero data in the interior of $\Omega _+$;
their mapping properties extend to the full scale of Sobolev spaces
with $s\in{\mathbb R}$.
$R_1$ and $R_2$ act in $L_2(\Omega _+)$ as the inverses of the
realizations 
$A_1$ resp.\ $A_2$ of $A$ with domains
\begin{equation}
D(A_1)=\{u\in H^2(\Omega _+)\mid \gamma _0u=0 \},\text{ resp.\ }
D(A_2)=\{u\in H^2(\Omega _+)\mid \nu _{A}u=0 \}.\label{tag2.8} 
\end{equation}
 
The operator $A_3$ representing the Robin condition $\nu
_Au=b\gamma _0u$ is defined similarly from the sesquilinear form 
\begin{equation}
s_b(u,v)=s(u,v)+(b\gamma _0u,\gamma _0v)_{L_2(\Sigma )}\label{tag2.9}
\end{equation}
on $H^1(\Omega _+)$ and has similar properties as $A_2$: its domain is
$D(A_3)=\{u\in H^2(\Omega _+)\mid (\nu _A-b\gamma _0)u=0 \}$, and the operator
\begin{equation}
{\cal A}_3=\begin{pmatrix} A\\\nu _A-b\gamma _0\end{pmatrix} \text{ has
  inverse }\begin{pmatrix} R_3&\; K_3\end{pmatrix}, \text{ with }R_3=Q_+-K_3(\nu _A-b\gamma _0)Q_+.\label{tag2.7a}
\end{equation}

The above facts have been known for many years, although the emphasis was not
always placed on including low values of $s$. Instead of accounting
for this aspect in detail here, we mention that the results are covered 
by the construction in the book Grubb
\cite{G96}, Chapter 3, and that the general $2m$-order case will be
treated below in Section \ref{Higher}.

\medskip

We shall now regard the realization defined by \eqref{tag1.5} from the point of
view of general nonlocal boundary value problems. The basic theory was
presented in Grubb  \cite{G68} and was taken up again and further
developed in a joint work with Brown and Wood \cite{BGW09};
applications to
exterior domains are included in \cite{G08}. (An introduction is also
given in \cite{G09}.) The fundamental result is
that the closed
realizations $\wA$ are in a 1--1 correspondence with the closed,
densely defined operators $T:V\to W$, where $V$ and $W$ are closed
subspaces of $Z$, the nullspace of $A_{\max}$. Many properties are
carried along in this correspondence, for example, $\wA$ is invertible
if and only if $T$ is so, and in the affirmative case one has 
the  Kre\u\i{}n-type formula
\begin{equation}
\wA^{-1}=A_1^{-1}+\inj_VT^{-1}\pr_W,\label{tag2.10}
\end{equation}
where $\inj_V$ denotes the injection $V\hookrightarrow H$ and 
$\pr_V$ denotes the orthogonal projection onto $V$, in $H=L_2(\Omega _+)$.
We have here taken the Dirichlet realization $A_1$ as the reference
operator for the correspondence theorem.

Consider in particular a realization $\wA$ corresponding to an
operator $T:Z\to Z$ (i.e., with $V=W=Z$).

As shown in the mentioned references, $\wA$ can be interpreted as 
representing a boundary
condition.
To describe that boundary condition, we first recall that \eqref{tag2.3} implies the Green's formula valid for
$u,v\in H^2(\Omega _+)$,
\begin{equation}
(Au,v)_{L_2(\Omega _+)}-(u,Av)_{L_2(\Omega _+)}=(\nu  _Au,\gamma _0v)_{L_2(\Sigma )}-
(\gamma _0u,\nu _Av)_{L_2(\Sigma )}.
\end{equation}
We denote by $\gamma _Z$  the restriction of $\gamma _0$ to $Z$,
\begin{equation}
\gamma _Z:Z\simto H^{-\frac12}(\Sigma ),\label{tag2.11}
\end{equation}
with adjoint $\gamma _Z^*:H^{\frac12}(\Sigma )\simto Z$ (recall that for $s\in{\mathbb
  R}$, $H^{-s}(\Sigma )$ identifies with the antidual (conjugate dual)
space $(H^s(\Sigma ))^*$ of $H^s(\Sigma )$, with a duality consistent
with the scalar product in $L_2(\Sigma )$). Moreover, we set
\begin{equation} P_{\gamma _0,\nu _{A}}=\nu _{A}K_1,\quad
\Gamma =\nu _{A}-P_{\gamma _0,\nu _{A}}\gamma _0,\text{ also equal to
}\nu _AA_1^{-1}\Ama;\label{tag2.12} 
\end{equation}
here $P_{\gamma _0,\nu _{A}}$ is a first-order elliptic pseudodifferential
operator  over $\Sigma $, and $\Gamma $ is a (nonlocal) trace operator.
There holds a generalized Green's formula for all $u,v\in D(A_{\max})$,
\begin{equation}
(Au,v)_{L_2(\Omega _+)}-(u,Av)_{L_2(\Omega _+)}=(\Gamma u,\gamma _0v)_{\frac12,-\frac12}-
(\gamma _0u,\Gamma v)_{-\frac12, \frac12},\label{tag2.13}
\end{equation}
 where $(\cdot,\cdot)_{s ,-s }$ indicates the (sesquilinear) duality pairing
between $H^s (\Sigma )$ and $H^{-s }(\Sigma
)$. 
The boundary condition that $\wA$ represents is then found to be
\begin{equation}
\Gamma u=L\gamma _0u,\label{tag2.14}
\end{equation}
 where $L$ is the closed, densely defined operator from
$H^{-\frac12}(\Sigma )$ to $H^\frac12(\Sigma )$ defined from $T$ by
\begin{equation}
L=(\gamma _Z^*)^{-1}T\gamma _Z^{-1}, \quad D(L)=\gamma _0D(T).\label{tag2.15}
\end{equation}
Since $\Gamma =\nu _{A}-P_{\gamma _0,\nu _{A}}\gamma _0$, the condition \eqref{tag2.14} can also be written
\begin{equation}
\nu _{A}u= (L+P_{\gamma _0,\nu _{A}})\gamma _0u,\label{tag2.16}
\end{equation}
so it is of the  form \eqref{tag1.5} with 
$C$ acting like $L+P_{\gamma _0,\nu _{A}}$.
To sum up:

\begin{proposition}\label{Proposition2.1} 
When $\wA$ corresponds to $T:Z\to Z$, it equals the
realization defined by the Neumann-type boundary condition {\rm \eqref{tag1.5}}, where
\begin{align}
C=L+P_{\gamma _0,\nu _{A}},& \quad L=(\gamma _Z^*)^{-1}T\gamma
_Z^{-1},\nonumber\\
D(C)=D(L)&=\gamma _0D(T).
\label{tag2.17}
\end{align}
\end{proposition}

Assume in the following that $0\in\varrho (\wA)$, equivalently $T$ has a bounded, everywhere defined inverse $T^{-1}:Z\to
Z$, and  $L$ has a bounded everywhere defined inverse $L^{-1}:H^{\frac12}(\Sigma )\to
H^{-\frac12}(\Sigma )$.
Then \eqref{tag2.10} takes the form:
\begin{equation}
\wA^{-1}=A_1 ^{-1}+\inj_ZT^{-1}\pr_Z=A_1 ^{-1}+K_1 L^{-1}K_1 ^*.\label{tag2.18}
\end{equation}
Here $K_1 $ is the Poisson
operator for the Dirichlet problem (cf.\ \eqref{tag2.6}), considered as a 
mapping from $H^{-\frac12}(\Sigma )$ to $L_2(\Omega _+)$ 
(also equal to $\inj_Z\gamma _Z^{-1}$); its
adjoint $K_1^*$ goes from $L_2(\Omega _+)$ to $H^{\frac12}(\Sigma )$. 
The formula \eqref{tag2.18} can clearly be used to examine $\wA^{-1}$ as a
perturbation of $A_1^{-1}$; we pursue this fact below in our analysis
of essential spectra.

\begin{remark}\label{Remark2.2}
We are interested in cases where $T$ has an essential
spectrum outside of 0. As a specific example, one can think of 
\begin{equation}
T=aI \text{ on }Z, \text{ with }a\in {\mathbb R}\setminus \{0\};\label{tag 2.19}
\end{equation}
its essential spectrum is $\{a\}$, since $\dim Z=\infty $.
In this case, 
\begin{equation}
L=a (\gamma _Z^*)^{-1}\gamma _Z^{-1}=a\Lambda_{(-1)}, \text{ where }\Lambda
_{(-1)}:H^{-\frac12}(\Sigma )\simto H^\frac12(\Sigma ) \label{tag2.20}
\end{equation}
is a pseudodifferential operator elliptic of order $-1$, and invertible. (This is in
contrast to those boundary conditions \eqref{tag1.5} that satisfy the
Shapiro-Lopatinski\u\i{} condition; they have $L$ elliptic of order
$+1$.) Since this $L$ is defined on all of $H^{-\frac12}(\Sigma )$, which is
mapped by $P_{\gamma _0,\nu _A}$ to $H^{-\frac32}(\Sigma )$, $C$ maps
$D(L)$ into $H^{-\frac32}(\Sigma )$; it is only the difference
$L=C-P_{\gamma _0,\nu _A}$ that is assured to map into
$H^{\frac12}(\Sigma )$.
The realization $\wA$ defined by this choice has $Z\subset D(\wA)$,
so $D(\wA)$ is not contained in $H^s(\Omega _+)$ for any $s>0$. It is
a variant of Kre\u\i{}n's ``soft extension''.  
\end{remark}

\section {Cutoff techniques}\label{Cutoff} 

 For the analysis of the operators on exterior domains we shall need
to study cutoffs, by multiplication either by a smooth function or by a
``rough'' characteristic function supported at a distance from the boundary.
In \cite{G84,G84a}, smooth cutoffs were used and the exterior singular
Green operators estimated by a
commutator argument based on a series of nested cutoff functions. We
shall here give a simpler argument based on rough cutoffs.

Let $\Omega _>$ be a smooth open subset of $\Omega _+$ such that $\comega_- \subset
\complement \comega_>$, and denote $\Omega _+\cap\complement
\comega_>=\Omega _< $. So $\Omega _+=\Omega _>\cup \Omega _<\cup
\partial\Omega _>$. We denote by $r^>$ resp.\ $r^<$ the
restriction operators from $\Omega _+$ to $\Omega _>$ resp.\ $\Omega _<$, and
by  $e^>$ resp.\ $e^<$ the
extension operators extending a function given on $\Omega _>$ resp.\
$\Omega _<$ to a function on $\Omega _+$ by zero on the complement in
$\Omega _+$.

In the following we draw on the analysis of singular numbers of
compact operators as presented in Gohberg and Kre\u\i{}n \cite{GK69}. The
operators lying in the intersection of
Schatten classes $\bigcap _{r>0}{\cal C}_r$ (also equal to $\bigcap
_{r>0}T_r$) will be called {\it
spectrally negligible}.

\begin{proposition}\label{Proposition3.1} Let $K_1$ be the Poisson operator entering
in {\rm \eqref{tag2.6}},
continuous from $H^{s-\frac12}(\Sigma )$ to $H^{s}(\Omega _+)$ for
all $s\in {\mathbb R}$, and consider the operators $K_{1,>}=r^>K _1:H^{-\frac12}(\Sigma
)\to L_2(\Omega _>)$ and $K_{1,>}^*=(r^>K_1)^*=K_1 ^*e^>: L_2(\Omega _>)\to H^{\frac12}(\Sigma
)$. Then $r^>K_1$ in fact  maps continuously
\begin{equation}
r^>K _1:  H^{s-\frac12}(\Sigma )\to H^{s'}(\Omega _>), \text{ any
}s,s'\in{\mathbb R}.\label{tag3.1} 
\end{equation}
Moreover, the operators  $K_{1,>}$ and $K_{1,>}^*$
are compact and
spectrally negligible.

Similar statements hold for $K_{j,>}=r^>K_j :H^{-\frac32}(\Sigma
)\to L_2(\Omega _>)$ and $K_{j,>}^*=K_j ^*e^>: L_2(\Omega _>)\to H^{\frac32}(\Sigma
)$ for $j=2,3$.

\end{proposition}

\begin{proof} Denote by $\gamma _0^>$ the operator
restricting to $\partial \Omega _>$.
When $\varphi \in
H^{-\frac12}(\Sigma )$, it follows by the interior regularity for
solutions of the Dirichlet problem for $A$ on $\Omega _+$ that $\gamma
^>_0K_1 \varphi \in C^\infty (\partial \Omega _>)$. 
Then $r^>K_1 \varphi $ is a null-solution
of the
Dirichlet problem for $A$ on $\Omega _>$ with $C^\infty $ boundary value.
This
will also hold if $\varphi \in H^{s-\frac12}(\Sigma )$, any $s\in{\mathbb
R}$. We know from the variational theory and regularity theory for the Dirichlet problem on
$\Omega _>$ that a null-solution with $C^\infty $ boundary value 
 lies in $H^{s'}(\Omega _>)$ for any $s'$; hence
\eqref{tag3.1} holds. It follows by duality that
\begin{equation}
K_1 ^*e^>:  (H^{s'}(\Omega _>))^*\to
H^{-s+\frac12}(\Sigma ), \text{ any }s',s\in{\mathbb R};\label{tag3.2}
\end{equation}
here $(H^{s'}(\Omega _>))^*=H^{-s'}(\Omega _>)$ when $|s'|<\frac12$
(generally it equals the space $H^{-s'}_0(\Omega _>)$ of
distributions in $H^{-s'}({\mathbb R}^n)$ supported in $\overline{\Omega} _>$).
Taking $s'=0$, we see that  
\begin{equation}
K_{1}^*e^>r^>K_{1}:  H^{s}(\Sigma )\to   H^{s''}(\Sigma ), \text{for
all $s,s''$,}\label{tag3.3}
\end{equation}
 so since $\Sigma $ is compact, this operator is compact (from $H^s(\Sigma
)$ to $H^{s'}(\Sigma )$, any $s,s'$), and lies in $\bigcap _{r>0}{\cal
C}_r$, i.e.\ is spectrally negligible. Then $K_{1,>}$ is compact from 
$H^{s}(\Sigma )$ to $L_2(\Omega _>)$
for any $s$, in particular for $s=\frac12$,
 so  $K_{1,>}K^*_{1,>}$ is compact in $L_2(\Omega _>$), and hence $K_{1,>}^*: L_2(\Omega _>)\to H^{\frac12}(\Sigma
)$ is compact. In view of the identity $s_l(K_{1,>}^*K_{1,>})=s_l(K_{1,>}K_{1,>}^*)
$, all $l$, all four operators are spectrally negligible.

The proofs for $K_{2,>}$ and $K_{3,>}$ follow the same
pattern.
\end{proof}

\begin{corollary}\label{Corollary3.2} Let $\eta \in C_0^\infty ({\mathbb R}^n,{\mathbb R})$ be
such that $\eta =1$ on a neighborhood of $\comega _>$. Then the operators $K_{j,\eta }=(1-\eta )K_j$ from $H^{-\frac12}(\Sigma
)$ to $ L_2(\Omega _+)$ for $j=1$, resp.\ from $H^{-\frac32}(\Sigma
)$ to $ L_2(\Omega _+)$ for $j=2,3$, map continuously
\begin{equation}
(1-\eta )K _j:  H^{s-\frac12}(\Sigma )\to H^{s'}(\Omega _+), 
\text{ any }s,s'\in{\mathbb R},\label{tag3.4}
\end{equation}
 and are
spectrally negligible.
\end{corollary}

\begin{proof} We can write $K_{j,\eta }=(1-\eta )K_j=e^>(1-\eta
)K_{j,>}$, where Proposition 3.1 applies to $K_{j,>}$,
and  $e^>(1-\eta )$ is bounded from $H^{s'}(\Omega _>)$ to
$H^{s'}(\Omega _+)$, any $s'$.
\end{proof}

\begin{corollary}\label{Corollary3.3} Consider the 
singular Green operators $G_j=-K_jT_jQ_+$ as in {\rm \eqref{tag2.7}, \eqref{tag2.7a} } with 
\begin{equation}
T_1=\gamma _0,\quad T_2=\nu _A, \quad T_3=\nu _A-b\gamma _0.\label{tag3.5}
\end{equation}
For $\eta $ as in Corollary {\rm 3.2}, the operators $(1-\eta )G_j$
are spectrally negligible.
\end{corollary}

\begin{proof} This follows since $T_jQ_+$ is bounded from $L_2(\Omega
_+)$ to $H^{\frac32}(\Sigma )$ for $j=1$, and from $L_2(\Omega
_+)$ to $H^{\frac12}(\Sigma )$ for $j=2,3$, and the $(1-\eta )K_j$ map
into $C^\infty $ and are spectrally negligible by Corollary 3.2.
\end{proof}

\begin{remark}\label{Remark3.4}
The proofs given above rely on the solvability properties of the
exterior problems for $A$. The properties can also be inferred  
from a
general principle shown in \cite{G96}, Lemma 2.4.8, on cutoffs of
Poisson operators, prepared for the  definition on admissible
manifolds (which include exterior domains).
Moreover, the lemma deals with a parameter-dependent pseudodifferential
boundary operator calculus, including a spectral parameter $\mu $. 
In this setting, when we consider the
Poisson operator family
$K_j^\lambda $ for $\{A-\lambda , T_j\}$, $\lambda $ on a ray
$\{\lambda =-\mu^2 e^{i\theta }\}$ in
${\mathbb C}\setminus {\mathbb R}_+$, it is of {\it regularity} $\nu
=+\infty $. Lemma 2.4.8 then implies that
$(1-\eta )K_j^\lambda $ is of order $-\infty$ and regularity $+\infty $, hence maps
$H^{s,\mu }(\Sigma )\to H^{s',\mu }(\Omega _+)$ for all $s,s'\in{\mathbb
R}$. Then $(K_j^\lambda )^*(1-\eta
)^2K_j^\lambda$
maps $H^{s,\mu }(\Sigma )\to H^{s'',\mu }(\Sigma )$ for all
$s,s''\in{\mathbb R}$. (The $H^{s,\mu }$-norms are based on the definition
of the norm on $H^{s,\mu }({\mathbb R}^n)$, namely $\|u\|_{s,\mu
}=\|{\cal F}^{-1}(\ang{(\xi,\mu )} ^s{\cal F}u)\|_{L_2({\mathbb R}^n)}$.) From this we can conclude both Corollary 3.2 and the
fact that any Schatten norm of $(1-\eta )K^\lambda _j$ is $O(\lambda ^{-N})$ (any $N$) for
$\lambda \to\infty $ on the ray, as first shown in
\cite{G84a}. Proposition 3.1 follows from this if we replace
$\eta $ by $\eta _1$ supported in $\complement \comega_>$ and equal to
1 on a neighborhood of $\Omega _-$; then $r^>K^\lambda _j=r^>(1-\eta
_1)K^\lambda _j$. Also here we get the rapid decrease in $\lambda $ of
the Schatten  norms.
\end{remark}

We use the results first to reprove the theorems of Birman \cite{B62} and
Birman-Solomiak \cite{BS80} with a slight elaboration, essentially as in
\cite{G84}, \cite{G84a}.

\begin{theorem}\label{Theorem3.5} For $j,k=1,2,3$, let
\begin{equation}
P_j=A_0^{-1}-A_j^{-1}\oplus 0_{L_2(\Omega _-)},\quad
G'_j=A_0^{-1}-A_j^{-1}\oplus (A_0^{-1})_-,\quad
G_{jk}=A_j^{-1}-A_k^{-1}.\label{tag3.6}\end{equation}
Then 
\begin{equation}
P_j\in T_{2/n},\quad G'_j\text{ and }G_{jk}\in T_{2/(n-1)}.\label{tag3.7}
\end{equation} 
Moreover, there are spectral asymptotics formulas for $l\to\infty $:
\begin{equation}
s_l(P_j)l^{2/n}\to C_0,\quad s_l(G_{jk})l^{2/(n-1)}\to C_{jk},\label{tag3.8}
\end{equation}
where the constants are determined from the principal symbols. Here
$C_0$ is the constant in the spectral asymptotics formula for
$(A_0^{-1})_-$, namely $C_0=\lim_{l\to\infty
}s_l((A_0^{-1})_-)l^{2/n}$, defined from the principal symbol of $A_0$
on $\Omega _-$. 
\end{theorem}

\begin{proof} 
We use the notation in \eqref{tag2.7}\,ff.\ and Corollary
3.3; in particular, $A_0^{-1}=Q$. It is well-known that $Q_-$ is
compact, with the asserted spectral asymptotics.

Consider first $G_{jk}$; in view of \eqref{tag2.7} it can be written
\begin{equation}
G_{jk}=-K_jT_jQ_++K_kT_kQ_+.
\end{equation}
Let $\eta $ be as in Corollary 3.2 and let
$\eta '\in C_0^\infty ({\mathbb R}^n,{\mathbb R})$, supported in a smooth bounded set $\Omega
'$ and with $\eta ' =1$ on a neighborhood
of $\supp\eta $. 
We can rewrite $-G_j=K_jT_jQ_+$ as follows:
\begin{equation}
K_jT_jQ_+=K_jT_j\eta Q_+=
\eta 'K_jT_j \eta Q_+\eta '+\eta 'K_jT_j \eta Q_+(1-\eta ')
 +(1-\eta ')K_jT_j \eta Q_+.\label{tag3.9}
\end{equation}
Here the first term is a singular Green operator on $\Omega '\cap
\Omega _+$ to which the calulus for bounded domains can be applied, and
the two other terms are  spectrally negligible. In  fact, $(1-\eta ')K_j
$ is so by Corollary 3.2, and for $\eta Q(1-\eta ')$ we can use
that it maps $H^s({\mathbb R}^n)$ continuously into
$H^{s'}(\Omega ')$ for all $s$ and $s'$, since $\supp\eta \cap \supp
(1-\eta ')=\emptyset$ so that the operator is of order $-\infty
$. Then since $\Omega '$ is bounded, the operator is spectrally
negligible, and so are its compositions with bounded operators.

The same arguments apply to $K_kT_kQ_+$, so we find that
\begin{equation}
G_{jk}= \eta '(-K_jT_j +K_kT_k)\eta Q_+\eta ' +{\cal R},
\end{equation}
where ${\cal R}$ is spectrally negligible and the first term is a
singular Green operator in $\Omega '\cap\Omega _+$. To the first term
we apply 
\cite{G84} Th.\ 4.10, which shows that this term is in $T_{2/(n-1)}$
and satisfies a spectral asymptotics formula as in \eqref{tag3.8}; these
facts are preserved when the spectrally negligible  term ${\cal R}$
is added on. This shows the assertions for the $G_{jk}$.

The
treatments of $K_jT_jQ_+$ in \cite{G84} (with misprints) and
\cite{G84a} are a bit more complicated in
their use of commutators  and nested sequences of cutoff functions. 

Next, consider $G'_j$. Here, since $Q_+\oplus 0=e^+r^+Qe^+r^+$ and
$0\oplus Q_-=e^-r^-Qe^-r^-$,
\begin{align*}
G'_j&=A_0^{-1}-A_j^{-1}\oplus (A_0^{-1})_-=Q-(Q_+-K_jT_jQ_+)\oplus Q_-\\
&=Q-Q_+\oplus Q_-+K_jT_jQ_+\oplus 0=e^+r^+Qe^-r^-+ e^-r^-Qe^+r^++K_jT_jQ_+\oplus 0.
\end{align*}
For $\tilde G=e^+r^+Qe^-r^-+ e^-r^-Qe^+r^+$ we proceed as in
\cite{G84}, Th.\ 5.1: Consider
\begin{equation}
\tilde G^2=e^+r^+Qe^-r^-Qe^+r^++ e^-r^-Qe^+r^+Qe^-r^-.\label{tag3.10}
\end{equation}
The second term acts like $0\oplus L_{\Omega _-}(Q,Q)$, where
$L_{\Omega _-}(Q,Q)=Q^2_--Q_-Q_-$ is the ``leftover operator'' for the
composition of $Q_-$ with $Q_-$; it is a singular Green operator 
and has a spectral
asymptotics formula with exponent $4/(n-1)$, by \cite{G84} Th.\ 4.10. (It
was in the quoted paper that the analysis of leftover operators in
terms of $e^+r^+Qe^-r^-$ and $e^-r^-Qe^+r^+$  was
first introduced.) 

The first term in \eqref{tag3.10} identifies similarly with a leftover operator 
on $\Omega
_+$, hence a singular Green operator, but since $\Omega _+$ is unbounded, we need more argumentation to
show that it is a compact operator with the desired spectral
asymptotics. With $\eta $ and $\eta '$ as above, we can write it:
\begin{align}
L_{\Omega _+}&(Q,Q)=L_{\Omega _+}(Q\eta ,\eta Q)\nonumber\\
=L_{\Omega _+}&(\eta 'Q\eta ,\eta Q\eta ')+L_{\Omega _+}((\eta 'Q\eta ,\eta Q(1-\eta '))+L_{\Omega _+}((1-\eta ')Q\eta ,\eta Q).
\label{tag3.11}
\end{align}
Here $\eta Q(1-\eta ')$ is spectrally negligible as noted above, and its adjoint
$(1-\eta ')Q\eta $ is likewise spectrally negligible. So $L_{\Omega
_+}(Q,Q)$ is the sum of a spectrally negligible part and $L_{\Omega _+}(\eta 'Q\eta ,\eta Q\eta ')$, a
singular Green operator in $\Omega '\cap \Omega _+$.

Thus $\tilde G^2=L_{\Omega _+}(\eta 'Q\eta ,\eta Q\eta ')\oplus
L_{\Omega _-}(Q,Q)$ plus spectrally negligible terms, so it follows from \cite{G84} Th.\ 4.10 that $\tilde G^2$ has a spectral
asymptotics behavior $s_l(\tilde G^2)l^{4/(n-1)}\to C$, and then
$\tilde G$ satisfies $s_l(\tilde G)l^{2/(n-1)}\to C^
\frac12$. 

We still have to include the term $K_jT_jQ_+\oplus 0$, but the
nontrivial part was already
treated further  above, and is seen to have a similar spectral 
asymptotics behavior.
Adding all contributions and using the rules for $s$-numbers, we find 
that $G'_j\in T_{2/(n-1)}$.

For $P_j$, we simply use that
\begin{equation}
P_j=G'_j+0_{L_2(\Omega _+)}\oplus Q_-,\label{tag3.12}
\end{equation}
where perturbation formulas as in \cite{G84} show that the spectral asymptotics formula for $Q_-$ dominates the
behavior. One could moreover give remainder estimates (as done in \cite{G84}).
\end{proof}

\begin{remark}\label{Remark3.6} The estimates also hold when $b$ for
  $A_3$ is
  replaced by a first-order differential operator $B$ such that the
  realization is elliptic and invertible. Related results are found for
$
G_{jk}^{(N)}=A_j^{-N}-A_k^{-N} 
$, which is a singular Green operator on $\Omega _+$ of the form of a
sum of Poisson operators composed with trace operators; this leads to
asymptotic estimates for all positive integers $N$:
$s_l(G_{jk}^{(N)})l^{2N/(n-1)}\to C^{(N)}_{jk}
$ for $l\to\infty $.
\end{remark}

\section {Perturbations}\label{Perturbations}

We shall now investigate the question of perturbations of the
essential spectrum.

When $f\in L_2(\Omega _+)$,
we also write $r^<f=f_<$, $r^>f=f_>$. 
Let us rewrite the action of $\wA^{-1}$ on $f\in L_2(\Omega )$ in terms
of its action on the parts $f_<$ and $f_>$, with matrix notation. When
$u=\wA^{-1}f$, we have that
\begin{equation}
u=\begin{pmatrix} u_<\\u_>\end{pmatrix}=\wA^{-1}\begin{pmatrix} f_<\\f_>\end{pmatrix}
=\begin{pmatrix} r^<\wA^{-1}e^< &\; r^<\wA^{-1}e^>\\r^>\wA^{-1}e^< &\; r^>\wA^{-1}e^>
\end{pmatrix}\begin{pmatrix} f_<\\f_>\end{pmatrix}.\label{tag4.1}
\end{equation}

Recalling \eqref{tag2.18}, we shall decompose the operators $A_1^{-1}$ and
$K_1LK_1^*$ in a similar way. For $A_1^{-1}$ we have:
\begin{align}
A_1 ^{-1}&=\begin{pmatrix} r^<A_1 ^{-1}e^< &\;r^<A_1 ^{-1}e^>\\r^>A_1 ^{-1}e^< & \;r^>A_1 ^{-1}e^>
\end{pmatrix} \label{tag4.2}\\
&=\begin{pmatrix} r^<A_1 ^{-1}e^< &\;r^<A_1 ^{-1}e^>\\r^>A_1 ^{-1}e^< &\; 0
\end{pmatrix} +\begin{pmatrix} 0 &\;0\\0 &\; r^>A_1 ^{-1}e^>
\end{pmatrix}. 
\nonumber
\end{align}
The entries in the first matrix are compact in $L_2$-norm, since
$r^<A_1 ^{-1}e^<$ maps $L_2(\Omega _<)$ into $ H^2(\Omega
_<)$ and  
$r^<A_1 ^{-1}e^>$ maps $L_2(\Omega _>)$ into $H^2(\Omega
_<)$, where the injection $H^2(\Omega
_<)\hookrightarrow L_2(\Omega _<)$ is compact, and  
 $e^>r^>A_1 ^{-1}e^<r^<$ is the adjoint of $e^<r^<A_1 ^{-1}e^>r^>$ in
 $L_2(\Omega _+)$. Since a compact perturbation leaves the essential
 spectrum invariant, the second
matrix has the same essential spectrum as $A_1^{-1}$, and we know from
Theorem \ref{Theorem3.5} that this equals $\sigma
_{\operatorname{ess}}A_0^{-1}$. In other words,
\begin{equation}
A_1^{-1}=0_{L_2(\Omega _<)}\oplus (r^>A_1 ^{-1}e^>)+ S_1,\label{tag4.3}
\end{equation}
where $S_1$ is compact in $L_2(\Omega _+)$ and  $\sigma
_{\operatorname{ess}}A_1^{-1}=\sigma
_{\operatorname{ess}}A_0^{-1}$.

Next, we write
\begin{align} 
K_1 L^{-1}&K_1 ^*=\begin{pmatrix} r^<K_1 L^{-1}K_1 ^*e^< &\; r^<K_1 L^{-1}K_1 ^*e^>\\r^>K_1 L^{-1}K_1 ^*e^< &\; r^>K_1 L^{-1}K_1 ^*e^>
\end{pmatrix} \label{tag4.4}\\
&= \begin{pmatrix} r^<K_1 L^{-1}K_1 ^*e^< &\; 0\\0 &\; 0
\end{pmatrix}+\begin{pmatrix} 0 &\; r^<K_1 L^{-1}K_1 ^*e^>\\r^>K_1 L^{-1}K_1 ^*e^< &\; r^>K_1 L^{-1}K_1 ^*e^>
\end{pmatrix}.\nonumber
\end{align}
In the last matrix, every nonzero element is the composition of a
bounded operator with either $r^>K_1$ or $K_1^*e^>$, hence is
spectrally negligible in view of Proposition \ref{Proposition3.1}. So this whole matrix
is spectrally negligible. In other words,
\begin{equation}
 K_1 L^{-1}K_1 ^*= (r^<K_1 L^{-1}K_1 ^*e^<) \oplus  0_{L_2(\Omega _>)}+S_2,\label{tag4.5}
\end{equation}
where $S_2$ is spectrally negligible. In particular, $ r^<K_1
L^{-1}K_1 ^*e^< \oplus 0_{L_2(\Omega _>)}$ has the same essential 
spectrum as $K_1 L^{-1}K_1^*$. 

Recall furthermore that 
\begin{equation*}
K_1 L^{-1}K^*_1=\inj_Z T^{-1}\pr_Z,
\end{equation*}
where $Z$ is infinite dimensional, hence 
\begin{equation}
\sigma
_{\operatorname{ess}}(r^<K_1 L^{-1}K_1 ^*e^<)\cup \{0\}=
\sigma
_{\operatorname{ess}}(K_1 L^{-1}K_1 ^*)=
\sigma
_{\operatorname{ess}}T^{-1}\cup\{0\}.\label{tag4.6}
\end{equation}

Adding \eqref{tag4.3} and
\eqref{tag4.5}, setting $S=S_1+S_2$, and observing that $0\in \sigma
_{\operatorname{ess}}A_0^{-1}$, $0\in \sigma
_{\operatorname{ess}}\wA^{-1}$ ($A_0$ and $\wA$ are unbounded
operators), we conclude:

\begin{theorem}\label{Theorem4.1} Let $\wA$ be as in Proposition {\rm \ref{Proposition2.1}}, and
assume that $0\in \varrho (\wA)$. Then  $\wA^{-1}$ can be
written  as the sum of a compact
operator $S$ in $L_2(\Omega _+)$ and an operator decomposed into a part acting in
$L_2(\Omega _<)$ and a part acting in $L_2(\Omega _>)$:
\begin{equation}
\wA^{-1}=(r^<K_1 L^{-1}K_1 ^*e^<)\oplus (r^>A_{1}^{-1}e^>)+S.\label{tag4.7}
\end{equation}
Here
\begin{equation}
\sigma
_{\operatorname{ess}}(r^<K_1 L^{-1}K_1 ^*e^<)\cup\{0\}=\sigma
_{\operatorname{ess}}T^{-1}\cup\{0\},\quad  \sigma
_{\operatorname{ess}}(r^>A_{1 }^{-1}e^>)\cup \{0\}= \sigma
_{\operatorname{ess}}A_{0}^{-1},\label{tag4.8}
\end{equation}
and hence
\begin{equation}
\sigma _{\operatorname{ess}}\wA^{-1}
=\sigma
_{\operatorname{ess}}T^{-1}\cup \sigma
_{\operatorname{ess}}A_{0}^{-1}.\label{tag4.9}
\end{equation}

\end{theorem}

Since the essential spectrum of $\wA$ itself is the reciprocal set of
the nonzero essential spectrum of $\wA^{-1}$, we also have:

\begin{corollary}\label{Corollary4.2} When $\wA$ is as in Theorem {\rm 4.1},
\begin{equation}
\sigma _{\operatorname{ess}}\wA=\sigma
_{\operatorname{ess}}A_0\cup\sigma _{\operatorname{ess}}T. \label{tag4.10}
\end{equation}
In particular, $\sigma _{\operatorname{ess}}\wA$ contains all points
of  $\sigma _{\operatorname{ess}}T$, and the points in 
$\sigma _{\operatorname{ess}}A_0$ cannot be removed from $\sigma _{\operatorname{ess}}\wA$.

\end{corollary}

The statements in the introduction follow: In case 1) we take $T$
as in Remark \ref{Remark2.2} of Section \ref{Description} in order to
add a point $\{a\}$; when $T$
acts like $aI$, $C$ acts
like $a\Lambda _{(-1)}+P_{\gamma _0,\nu _A}$. A general
choice of a selfadjoint invertible $T_0$ in a separable
infinite-dimensional Hilbert space $Z_0$
gives rise to a selfadjoint invertible $T$ in the Hilbert space $Z$
with the same essential spectrum, by use of a unitary operator from 
$Z_0$ to $Z$. The statement
in 2) follows since we have covered all possibilities for $T$ in the
case of Neumann-type boundary conditions.

\section{Higher order cases}\label{Higher}

Similar results can be shown for higher order elliptic operators. 
The selfadjoint strongly elliptic even-order case is the natural
generalization of the case considered in the preceding sections; here
there is a solvable Dirichlet problem, and a selfadjoint invertible realization defined by
another boundary condition can be related to the Dirichlet realization
by a Kre\u\i{}n-type formula generalizing \eqref{tag2.18} 
as in \cite{G68,G74,BGW09}.

Invertible realizations exist in greater generality, though, so to
save later repetitions, we consider to begin with a 
more general
class assuring existence of resolvents $(\wA-\lambda )^{-1}$ at least when 
$\lambda $ is large, lying in a suitable subset of ${\mathbb C}$. 
We take for $A$ an elliptic operator $A=\sum_{|\alpha |\le
2m}a_\alpha (x)D^\alpha $ of order $2m$, $m$ integer,
with complex $C^\infty $ coefficients on ${\mathbb R}^n$ that are bounded with
bounded derivatives and with the principal symbol $a^0(x,\xi )=
\sum_{|\alpha |=
2m}a_\alpha \xi ^\alpha $ satisfying (with $c_1>0$) 
\begin{equation}
\operatorname{Re}a^0(x,\xi )\ge c_1|\xi |^{2m}, \text{ for }x,\xi \in{\mathbb
R}^n;\label{tag5.1}
\end{equation}
uniform strong ellipticity.  Here $D^\alpha =D_1^{\alpha _1}\cdots
D_n^{\alpha _n}$, $D_j=-i\partial/\partial x_j$. Denote by $A_0$ the maximal realization
of $A$ in $L_2({\mathbb R}^n)$; the uniform ellipticity implies that
$D(A_0)=H^{2m}({\mathbb R}^n)$. 

$A$ satisfies a G\aa{}rding inequality, for
which we include a quick proof:

\begin{lemma}\label{Lemma5.1} There are constants $c_0>0$ and $k\in{\mathbb
R}$ such that
\begin{equation}
\operatorname{Re}(Au,u)\ge c_0\|u\|_{m}^2-k\|u\|_{0}^2,\text{ for all
}u\in 
H^{2m}({\mathbb R}^n).\label{tag5.2}
\end{equation}

\end{lemma}

\begin{proof} Using the calculus of globally estimated
pseudodifferential operators as in \cite{H85} Sect.\ 18.1, and
\cite{G96}, we can write 
\begin{equation*}
A_1=\Lambda ^{-m}A\Lambda ^{-m},\quad \operatorname{Re}A_1=\tfrac12(A_1+A_1^*)=P^*P+B,
\end{equation*}
 where $\Lambda ^s=\Op(\ang\xi ^s)$, $A_1$ is of order 0 with
principal symbol $a_1^0(x,\xi )$ satisfying 
\begin{equation*}
\operatorname{Re}a_1^0(x,\xi )\ge c'_1>0,
\end{equation*}
$P$ is of order 0 with principal symbol
$p^0=(\operatorname{Re}a^0_1)^{\frac12}$, and $B$ is of order
$-1$. Since $P$ is elliptic, it has a parametrix $Q$ of order 0 so
that $I-QP$ is of order $-1$; hence 
\begin{align*}
\|v\|_0^2&=\|QPv+(I-QP)v\|_0^2\le C\|Pv\|_0^2+C'\|v\|_{-1}^2\\
&=C(P^*Pv,v)+C'\|v\|_{-1}^2\le C\operatorname{Re}(A_1v,v)+C''\|v\|_{-\frac12}^2,
\end{align*}
for $v\in {\cal S}({\mathbb R}^n)$ (the Schwartz space of rapidly
decreasing functions, dense in any $H^s({\mathbb R}^n)$).
It follows that (with $\|u\|_s=\|\Lambda ^su\|_{L_2({\Bbb R}^n)}$) 
\begin{align*}
\operatorname{Re}(Au,u)&=\operatorname{Re}(A_1\Lambda ^{m}u,\Lambda ^m
u)\\
&\ge C^{-1}\|u\|^2_m-C^{-1}C''\|u\|_{m-\frac12}^2
\ge \tfrac12 C^{-1}\|u\|_m^2-k\|u\|_0^2,
\end{align*}
where we used that $\|u\|^2_{m-\frac12}\le \varepsilon \|u\|_m^2+C(\varepsilon
)\|u\|_0^2$, any $\varepsilon >0$.
  \end{proof}

Since $|(Au,u)|\le C_1\|u\|_m^2$ and $\|u\|_m\ge \|u\|_0$, we can
infer from \eqref{tag5.2} that 
\begin{align*}
|\operatorname{Im}(Au,u)|&\le |(Au,u)|\le C_1\|u\|_m^2\le
C_1c_0^{-1}(\operatorname{Re}(Au,u)+k\|u\|_0^2)\\
\operatorname{Re}(Au,u)&\ge c_0\|u\|_0^2-k\|u\|_0^2=(c_0-k)\|u\|_0^2
;
\end{align*}
hence the numerical range of $A_0$, $\nu (A_0)=\{(A_0u,u)/\|u\|_0^2\mid u\in
D(A_0)\setminus \{0\}\}$,  
 is contained in a sectorial region $V$, 
\begin{equation}
\nu (A_0)\subset V\equiv
\{\lambda \in{\mathbb C} \mid \operatorname{Re} \lambda \ge c_0-k,\, |\operatorname{Im}\lambda |\le
c_2 (\operatorname{Re}\lambda +k)\},\label{tag5.3}
\end{equation} with $c_2=C_1c_0^{-1}$.
The numerical
range of the adjoint $A_0^*$ is likewise contained in $V$, and $V$ contains the
spectrum of $A_0$. (The elementary functional analysis used here is
explained e.g.\ in \cite{G09}, Ch.\ 12.)
 
For simplicity we add $kI$ to $A$, so that we can use the information
with $k=0$ in the following, replacing $V$ by 
\begin{equation}
V_0=
\{\lambda \in{\mathbb C}
\mid \operatorname{Re }\lambda \ge c_0,\, |\operatorname{Im}\lambda |\le
c_2 \operatorname{Re}\lambda\}.\label{tag5.4}
\end{equation}

The Dirichlet trace $\gamma u$ is in the $2m$-order case defined by \begin{equation*}
\gamma u=\{\gamma
_0u,\dots, \gamma  _{m-1}u\},
\end{equation*} with $\gamma _ju=\gamma
_0(\sum\nu _kD_k)^ju$. For the Dirichlet problems on smooth exterior or interior
subsets of ${\mathbb R}^n$, the variational construction gives a
realization with numerical range and spectrum likewise contained 
in $V_0$. Moreover, there are
Sobolev space mapping properties of the solution operator; this is extremely
well-known 
for bounded domains, and for exterior domains it is covered e.g.\ by
Cor.\ 3.3.3 of \cite{G96} (the differential operator $A-\lambda $ is uniformly
parameter-elliptic on all rays $\{\lambda =re^{i\theta }\mid r\ge 0\}$ outside 
$V_0$, and 
parameter-ellipticity of the boundary problem holds uniformly for $x'$
in the boundary).

Let us specify the result for $\Omega _+$ and $\Sigma $ defined
as in Section \ref{Introduction}. We denote 
$A_0^{-1}=Q$. Then
\begin{equation}
{\cal A}_\gamma =\begin{pmatrix} A\\\gamma \end{pmatrix} : 
H^{s+2m}(\Omega _+) \to
\begin{matrix}
H^{s}(\Omega _+)\\ \times \\ {\prod} _{0\le j<m}H^{s+2m-j-\frac12}(\Sigma  )\end{matrix}
\label{tag5.5}
\end{equation}
has for $s>-m-\frac12$ the
solution operator, continuous in the opposite direction,
\begin{equation}
{\cal A}_\gamma^{-1}=\begin{pmatrix} R_\gamma& \; K_\gamma\end{pmatrix},\text{ with }
R_\gamma=Q_+-K_\gamma\gamma Q_+.
 \label{tag5.6}
\end{equation}
Here $R_\gamma $ is the inverse of the Dirichlet realization $A_\gamma
$, which acts
like $A$ with domain $D(A_\gamma )=H^{2m}(\Omega _+)\cap
H^m_0(\Omega _+)$.

The general theory of \cite{G68,BGW09} is here interpreted by use of 
the Poisson operator $K_\gamma:{\prod} _{0\le j<m}H^{-j-\frac12}(\Sigma  )\to L_2(\Omega
_+)$ (and variants with $\lambda $-dependence). $K_\gamma $ acts as an inverse of 
\begin{equation}
\gamma _Z:Z\simto {\prod} _{0\le j<m}H^{-j-\frac12}(\Sigma  ),\label{tag5.7}
\end{equation}
$Z$ denoting the $L_2(\Omega _+)$ nullspace of $A$. The formulas are
exactly the same as in \cite{BGW09} Section 3.3.

We shall compare $A_\gamma $ with the realization $A_{B\varrho }$ of a
general {\it normal} boundary condition, defined as in \cite{G71}, \cite{BGW09} (3.85). Let 
\begin{equation*}M=\{0,1,\dots,2m-1\},\text{ denoting  }\varrho u =\{\gamma _ju\}_{j\in
M};\end{equation*} 
the Cauchy data. Let
$J$ be a subset of $M$ with $m$ elements, and let $B$ be a $J\times
M$-matrix of differential operators $B_{jk}$ on $\Sigma $ of order $j-k$:
\begin{equation}
B=(B_{jk})_{j\in J,k\in M} \text{ with }B_{jk}=0 \text{ for }k>j,\, B_{jj}=I. 
\label{tag5.8}\end{equation}
The boundary condition \cite{BGW09} (3.85): $\gamma
_ju+\sum_{k<j}B_{jk}\gamma _ku=0$ for $j\in J$,  can then be written
\begin{equation*}
B\varrho u=0;
\end{equation*}
it defines the realization $A_{B\varrho }$ with domain 
\begin{equation*}
D(A_{B\varrho })=\{u\in D(\Ama)\mid B\varrho u=0\}.
\end{equation*}
 Special examples are the
cases where $J=M_0$ or $M_1$,
\begin{equation*}
M_0=\{0,1,\dots, m-1\}, \quad M_1=\{m,m+1,\dots, 2m-1\},\text{
denoting }\nu u=\{\gamma _ju \}_{j\in M_1};
\end{equation*}
they define Dirichlet-type resp.\ Neumann-type conditions.

Let us assume that 
$\{A-\lambda , B\varrho \}$ is uniformly parameter-elliptic for 
$\lambda $ on a ray outside
$V_0$; then for large $\lambda $ on the ray, $\{A-\lambda , B\varrho
\}$ is invertible. Take a $\lambda _0$ where this invertibility holds
(assuming invertibility of $A_0-\lambda _0$ and $A_\gamma -\lambda _0$
also), and
denote in the rest of this section $A-\lambda _0$ by $A$;
then we are in the situation where
\begin{equation}
{\cal A}_{B\varrho } =\begin{pmatrix} A\\ B\varrho  \end{pmatrix} : 
H^{s+2m}(\Omega _+) \to 
\begin{matrix}
H^{s}(\Omega _+)\\ \times \\ {\prod} _{j\in J}H^{s+2m-j-\frac12}(\Sigma  )\end{matrix}
\label{tag5.9}
\end{equation}
for $s>-\frac12$ has the
solution operator, continuous in the opposite direction,
\begin{equation}
{\cal A}_{B\varrho }^{-1}=\begin{pmatrix} R_{B\varrho }& \;K_{B\varrho }\end{pmatrix},\text{ with }
R_{B\varrho }=Q_+-K_{B\varrho }B\varrho  Q_+.
 \label{tag5.10}
\end{equation}
Here $R_{B\varrho }$ is the inverse of the realization $A_{B\varrho }
$, which acts like $A$ with
domain $D(A_{B\varrho } )=\{u\in H^{2m}(\Omega _+)\mid B\varrho u=0\}.$

The difference between $A_\gamma ^{-1}$ and $A_{B\varrho }^{-1}$, and
more generally between two solution operators $A_{B\varrho }^{-1}$ and
$A_{\tilde B\varrho }^{-1}$, can
be described spectrally very much like in Section \ref{Cutoff}. First there is a
generalization of Proposition \ref{Proposition3.1} and its corollaries. 
Define $\Omega _\gtrless$,
$r^\gtrless$ and $e^\gtrless$ as in Section \ref{Cutoff}.

\begin{proposition}\label{Proposition5.2} 

$1^\circ$ The operators $K_{B\varrho ,>}=r^>K
_{B\varrho }:{\prod}_{j\in J}H^{-j-\frac12}(\Sigma
)\to L_2(\Omega _>)$ and $(K_{B\varrho ,>})^*=K_{B\varrho } ^*e^>:
L_2(\Omega _>)\to {\prod} _{j\in J}H^{j+\frac12}(\Sigma
)$  map continuously
\begin{align}
r^>K _{B\varrho }&:  {\prod}_{j\in J}H^{s-j-\frac12}(\Sigma )\to
H^{s'}(\Omega _>), \text{ any }s,s'\in{\mathbb R},\nonumber\\
K_{B\varrho } ^*e^>&:  (H^{s'}(\Omega _>))^*\to
{\prod} _{j\in J}H^{-s+j+\frac12}(\Sigma ), \text{ any }s',s\in{\mathbb R},
\label{tag5.11}
\end{align}
and are spectrally negligible.

$2^\circ$ When $\eta $ is a function in
$C_0^\infty ({\mathbb R}^n,{\mathbb R})$ that is $1$ on a neighborhood of
$\comega_>$, the operators $(1-\eta )K
_{B\varrho }:{\prod}_{j\in J}H^{-j-\frac12}(\Sigma
)\to L_2(\Omega _+)$ and $K_{B\varrho } ^*(1- \eta ):
L_2(\Omega _+)\to {\prod} _{j\in J}H^{j+\frac12}(\Sigma
)$  map continuously
\begin{align}
(1-\eta )K _{B\varrho }&:  {\prod}_{j\in J}H^{s-j-\frac12}(\Sigma )\to
H^{s'}(\Omega _+), \text{ any }s,s'\in{\mathbb R},\nonumber\\
K_{B\varrho } ^*(1-\eta )&:  (H^{s'}(\Omega _+))^*\to
{\prod} _{j\in J}H^{-s+j+\frac12}(\Sigma ), \text{ any }s',s\in{\mathbb R},
\label{tag5.12}
\end{align}
and are spectrally negligible.

\end{proposition}

\begin{proof} Denote by $\gamma ^>$ the Dirichlet trace operator for
$2m$-order operators on $\Omega _>$.
When $\varphi \in
{\prod}_{j\in J}H^{-j-\frac12}(\Sigma )$, $K_{B\varrho }\varphi $ is
$C^\infty $ on $\Omega _+$, hence $\gamma
^>K_{B\varrho } \varphi \in C^\infty (\partial \Omega _>)$. 
Then $r^>K_{B\varrho } \varphi $ is a null-solution
of the
Dirichlet problem for $A$ on $\Omega _>$ with $C^\infty $ boundary value.
This
will also hold if $\varphi \in H^{s-\frac12}(\Sigma )$, any $s\in{\mathbb
R}$. We now use that the Dirichlet problem on
$\Omega _>$ has a solution operator with mapping properties similar to
the problem for $\Omega _+$, and the proof is completed in the same way as
the proofs of Proposition \ref{Proposition3.1} and Corollary
\ref{Corollary3.2}. \end{proof}

\begin{remark}\label{Remark5.3} It may be observed as in Remark 
  \ref{Remark3.4} of Section \ref{Cutoff} that the
proofs can also be inferred from \cite{G96}, Lemma 2.4.8, and this
moreover implies a rapid decrease in $\lambda $ of any Schatten norm.
\end{remark}

With Proposition \ref{Proposition5.2} it is easy to generalize Theorem \ref{Theorem3.5} as follows:

\begin{theorem}\label{Theorem5.4} For $B\varrho $ and $\tilde B\varrho $ as above,
defining invertible elliptic realizations, let
\begin{equation}
P=A_0^{-1}-A_{B\varrho }^{-1}\oplus 0_{L_2(\Omega _-)},\quad
G'=A_0^{-1}-A_{B\varrho }^{-1}\oplus (A_0^{-1})_-,\quad
G''=A_{B\varrho }^{-1}-A_{\tilde B\varrho }^{-1}.\label{tag5.13}
\end{equation}
Then 
\begin{equation}
P\in T_{2m/n},\quad G'\text{ and }G''\in T_{2m/(n-1)}.\label{tag5.14}
\end{equation} 
Moreover, there are spectral asymptotics formulas for $l\to\infty $:
\begin{equation}
s_l(P)l^{2m/n}\to C,\quad s_l(G'')l^{2m/(n-1)}\to C'';\label{tag5.15}
\end{equation}
where the constants are determined from the principal symbols. Here
$C$ is the same constant as for $(A_0^{-1})_-$, namely $C=\lim_{l\to\infty }s_l((A_0^{-1})_-)l^{2m/n}$. 
\end{theorem}

\begin{proof} We proceed as in the proof of Theorem \ref{Theorem3.5}. First,
\begin{equation*}
G''=-K_{B\varrho }B\varrho  Q_++K_{\tilde B\varrho }\tilde B\varrho  Q_+
\end{equation*}
is written as a singular Green operator on $\Omega '\cap \Omega _+$
plus a spectrally negligible term, by a version of \eqref{tag3.9} applied to
both terms. The assertions for $G''$ then follow from \cite{G84}
Th.\ 4.10. Next, $G'$ is treated similarly to $G'_j$ in Theorem
\ref{Theorem3.5}, noting that the operators of order 2 have been replaced by
operators of order $2m$. Finally, $P= G'+0_{L_2(\Omega _+)}\oplus Q_-$,
where the spectral asymptotics behavior of $Q_-$ dominates the sum, in
view of the rules for $s$-numbers. 
\end{proof}

It is here allowed to take the set $J$ for $B\varrho $ different from the
corresponding set $\tilde J$ for $\tilde B\varrho $. There are similar
results for differences between higher powers $A_{B\varrho
}^{-N}-A_{\tilde B\varrho }^{-N}$, as in Remark \ref{Remark3.6} of
Section \ref{Cutoff}.

A result of the type  $A_{B\varrho }^{-N}-A_{\tilde B\varrho }^{-N}\in
T_{2mN/(n-1)}$ has been announced by Gesztesy and Malamud in \cite{GM08}, apparently based on a
consideration of $M$-functions.

In all the calculations, $A$ can be taken to be a $(p\times
p)$-system, acting on $p$-vectors. When $A$ is scalar, the 
boundary conditions with \eqref{tag5.8} are the most general ones for which
parameter-ellipticity can hold (cf.\ \cite{G96}, Sect.\ 1.5); in the 
systems case there exist more
general normal boundary conditions, as studied in \cite{G74}. The above
analysis can be extended to include these, mainly at the cost of a more
complicated notational apparatus.  Pseudodifferential $B_{jk}$ could be
allowed as in \cite{G96}.

\begin{remark} \label{Remark5.2}For bounded domains, the result for $G''$ has been known
  since 1984, since $A_{B\varrho }^{-1}-A_{\tilde B\varrho }^{-1}$ is
  then itself a singular Green operator of order $-2m$ on a bounded
  domain, to which \cite{G84} Th.\ 4.10 applies. For selfadjoint
  cases, see also \cite{G74} Sect.\ 8.
\end{remark}

\section {Spectral perturbations in higher-order cases}\label{Spectral}

For the study of perturbations of essential spectra we restrict the
attention to selfadjoint realizations. First of all, this requires
that $A$ equal its formal adjoint $A'$, moreover, it restricts the
sets $J$ and matrices $B$ that can be allowed. With the notation
$N'=\{k\mid 2m-k-1\in N\}$, we have as a necessary
condition on $J$ is that it should equal its reversed complement:
\begin{equation}
J=K',\text{ where }K=M\setminus J.\label{tag6.1}
\end{equation}
To explain further, we recall some details from \cite{G71}. From the Green's formula 
\begin{equation*}
(Au,v)-(u,Av)=({\cal A}\varrho
u,\varrho v)=\Big(\begin{pmatrix} {\cal A}_{M_0M_0}&{\cal
A}_{M_0M_1}\\{\cal A}_{M_1M_0}&0\end{pmatrix} \begin{pmatrix} \gamma u\\\nu
u\end{pmatrix}, \begin{pmatrix} \gamma u\\\nu u\end{pmatrix}\Bigr),
\end{equation*}
where ${\cal A}$ is skew-selfadjoint and invertible, it is seen that when we set 
\begin{equation}
\chi u={\cal A}_{M_0M_1}\nu u+\tfrac12 {\cal A}_{M_0M_0}\gamma u,\label{tag6.2}
\end{equation}
(taking $\frac12$ of the contribution from ${\cal A}_{M_0M_0}$ along), we get the symmetric formula
\begin{equation}
(Au,v)_{L_2(\Omega _+)}-(u,Av)_{L_2(\Omega _+)}=(\chi u,\gamma
v)_{L_2(\Sigma )^m} -(\gamma  u,\chi  v)_{L_2(\Sigma )^m},
\label{tag6.3}
\end{equation}
valid for $u,v\in H^{2m}(\Omega _+)$. Here $\chi $ is indexed by
$M_0$, $\chi =\{\chi _j\}_{j\in M_0}$ with
$\chi _j$ of order $2m-j-1$; it replaces $\nu $ in systematic
considerations and maps from $H^s(\Omega _+)$ to ${\prod}_{j\in
  M_0}H^{s-2m+j+\frac12}(\Sigma )$.  Green's formula
 has the extension to $u\in D(\Ama)$, $v\in H^{2m}(\Omega _+)$:
\begin{equation*}
(Au,v)_{L_2(\Omega _+)}-(u,Av)_{L_2(\Omega _+)}=(\chi  u,\gamma
v)_{\{-2m+j+\frac12\},\{2m-j-\frac12\}} -(\gamma  u,\chi 
v)_{\{-j-\frac12\},\{j+\frac12\}},
\end{equation*}
where $(\cdot,\cdot)_{\{-s_j\},\{s_j\}}$ denotes the duality between
${\prod} H^{-s_j}(\Sigma )$ and ${\prod} H^{s_j}(\Sigma )$.
With $\chi $ replaced by the ``reduced Neumann trace operator'' $\Gamma $,
one has for $u,v\in D(\Ama)$:
\begin{equation}
(Au,v)_{L_2(\Omega _+)}-(u,Av)_{L_2(\Omega _+)}=(\Gamma  u,\gamma
v)_{\{j+\frac12\},\{-j-\frac12\}} -(\gamma  u,\Gamma
v)_{\{-j-\frac12\},\{j+\frac12\}};
\label{tag6.4}
\end{equation}
here 
\begin{equation}
P_{\gamma ,\chi }=\chi K_\gamma ,\quad \Gamma =\chi -P_{\gamma ,\chi }\gamma =\chi A_\gamma ^{-1}\Ama.\label{tag6.5}
\end{equation}

Now when $J$ satisfies \eqref{tag6.1}, the subsets
\begin{equation*}
J_0=J\cap M_0, \quad J_1=J\cap M_1,\quad K_0=K\cap M_0,\quad K_1=K\cap M_1,
\end{equation*}
satisfy
\begin{equation}
{K_1}'=J_0,\quad {J_1}'=K_0.\label{tag6.6}
\end{equation}
We set 
\begin{equation*}
\gamma _{J_0}=\{\gamma _j\}_{j\in J_0},\quad \gamma _{K_0}=\{\gamma
_j\}_{j\in K_0},\quad \chi _{J_0}=\{\chi _j\}_{j\in J_0},\quad \chi _{K_0}=\{\chi _j\}_{j\in K_0}.
\end{equation*}
As shown in \cite{G71} and recalled in \cite{BGW09}, the boundary
condition $B\varrho u=0$ may then be rewritten in the form, with
differential operators $F_0$, $G_1$, $G_2$,
\begin{equation}
\gamma _{J_0}u=F_0\gamma _{K_0}u,\quad \chi _{K_0}u=G_1\gamma _{K_0}u+G_2\chi _{J_0}u,\label{tag6.7} 
\end{equation}
when we take \eqref{tag6.6} into account. Here the first condition
$\gamma _{J_0}u=F_0\gamma _{K_0}u$ can be viewed as the ``Dirichlet
part'', purely concerned with $\gamma u$, whereas the second condition
$\chi _{K_0}u=G_1\gamma _{K_0}u+G_2\chi _{J_0}u$ can be viewed as the
``Neumann-type part'', where part of the Neumann data $\chi _{K_0}u$
is given as a function of the other data. Note that $G_1$ links
the free Dirichlet data $\gamma _{K_0}u$ to Neumann data and has
entries of positive order, and $G_2$ has entries of order $<m$.

The boundary condition for the
adjoint realization is then
\begin{equation}
\gamma _{J_0}u=-G_2^*\gamma _{K_0}u,\quad \chi _{K_0}u=G_1^*\gamma
_{K_0}u-F_0^*\chi _{J_0}u. \label{tag6.8}
\end{equation}

If $J=M_0$, the condition $B\varrho u=0$ reduces to the Dirichlet
condition $\gamma u=0$. To get a different condition we must take
$J\ne M_0$; this means that $K_0\ne \emptyset$. 

We assume in the following that $\{A-\lambda , B\varrho \}$ is
uniformly parameter-elliptic on a ray outside $V_0$ as in the preceding section,
so that $D(A_{B\varrho })\subset H^{2m}(\Omega _+)$. Then 
\begin{equation}
G_2^*=-F_0,\quad G_1^*=G_1,\label{tag6.9}
\end{equation}
is necessary and sufficient for selfadjointness of $A_{B\varrho }$.
\eqref{tag6.9} is assumed from now on.

The operator $A_{B\varrho }$ corresponds to a selfadjoint
 operator $T:V\to V$ by
the general theory, where $V$ is the $L_2(\Omega _+)$-closure of
$\pr_\gamma D(A_{B\varrho })$ (here $\pr_\gamma =I-A_\gamma
^{-1}\Ama$). $V$ is mapped by $\gamma $ onto
the closure $X$ of $\gamma D(A_{B\varrho })$ in ${\prod} _{k\in
M_0}H^{-k-\frac12}(\Sigma ) $. Here $X$ is the graph of $F_0$, so it is homeomorphic to ${\prod} _{k\in
K_0}H^{-k-\frac12}(\Sigma ) $, by the mappings
\begin{align}
\Phi =\begin{pmatrix} I_{K_0K_0}\\ F_0\end{pmatrix},& \quad \pr_1=\begin{pmatrix}
I&\; 0\end{pmatrix},\label{tag6.10}\\
 \Phi: {\prod}_{k\in K_0
}H^{-k-\frac12}(\Sigma )\simto X,&\quad \pr_1 :X\simto {\prod}_{k\in K_0
}H^{-k-\frac12}(\Sigma ),
\end{align}
as shown in \cite{G71} and recalled in \cite{BGW09}. Here $V=K_\gamma X=K_\gamma \Phi {\prod}_{k\in K_0
}H^{-k-\frac12}(\Sigma )$.

The restriction
of $\gamma $ to a mapping from
$V$ to $X$ is denoted $\gamma _V$, so we have:
\begin{equation*}
\gamma _V:V\simto X,\; \pr_1\gamma _V:V\simto {\prod}_{k\in K_0
}H^{-k-\frac12}(\Sigma ),\; \gamma _V^{-1}\Phi :{\prod}_{k\in K_0
}H^{-k-\frac12}(\Sigma )\simto V.
\end{equation*}
With these definitions, \eqref{tag6.7} may be written (using \eqref{tag6.9})
\begin{equation}
\gamma u=\Phi \gamma _{K_0}u,\quad \Phi ^* \chi u= G_1\gamma
_{K_0}u.\label{tag6.11} 
\end{equation}

The operator $T$
in $V$ is carried over to an operator 
\begin{equation}
L=(\gamma _V^*)^{-1}T\gamma _V^{-1}:X\to X^*,\label{tag6.12}
\end{equation} 
which is further
translated to an operator 
\begin{equation}
L_1=\Phi ^*L\Phi :{\prod}_{k\in K_0
}H^{-k-\frac12}(\Sigma )\to {\prod}_{k\in K_0
}H^{k+\frac12}(\Sigma ).\label{tag6.13}
\end{equation}
We now recall from \cite{G71} how the form of $L_1$ is determined (this
detail was not repeated in \cite{BGW09}). Consider  the condition defining the
correspondence between $A_{B\varrho }$ and $T$ (cf.\ \cite{G68,G71,BGW09}):
\begin{equation}
(Au,z)=(T\pr_\zeta u,z)\text{ for all }u\in D(A_{B\varrho }), \; z\in V.\label{tag6.14}
\end{equation}
Here the right-hand side is rewritten as 
\begin{equation*}
(T\pr_\zeta u,z)=(L\gamma u,\gamma z)_{\{k+\frac12\},\{-k-\frac12\}}=(L_1 \gamma _{K_0}u,\gamma _{K_0}z)_{\{k+\frac12\},\{-k-\frac12\}},
\end{equation*}
whereas the left-hand side takes the form, in view of \eqref{tag6.4} and \eqref{tag6.11}:
\begin{align*}
(Au,z)&=(\Gamma u,\gamma z)=(\chi u-P_{\gamma ,\chi }\gamma u, \Phi
\gamma _{K_0}z)_{\{k+\frac12\},\{-k-\frac12\}} \\
&=(\Phi ^*\chi u-\Phi ^*P_{\gamma ,\chi }\Phi \gamma
_{K_0}u,\gamma _{K_0}z)_{\{k+\frac12\},\{-k-\frac12\}}\\
&= ((G_1 -\Phi ^*P_{\gamma ,\chi }\Phi )\gamma _{K_0}u,\gamma _{K_0}z)_{\{k+\frac12\},\{-k-\frac12\}}.
\end{align*}
Since $\gamma _{K_0}z$ runs in a dense subset of ${\prod}_{k\in K_0
}H^{-k-\frac12}(\Sigma )$, \eqref{tag6.14} implies
\begin{equation}
L_1\gamma _{K_0}u=(G_1-\Phi ^*P_{\gamma ,\chi }\Phi )\gamma _{K_0}u,\label{tag6.15}
\end{equation}
so $L_1$ acts like $G_1-\Phi ^*P_{\gamma ,\chi }\Phi$. The boundary
condition may then be rewritten as  
\begin{equation}
\gamma u=\Phi \gamma _{K_0}u,\quad \Phi ^* \chi u= (L_1+\Phi ^*P_{\gamma ,\chi }\Phi )\gamma _{K_0}u.\label{tag6.16}
\end{equation}
Since $\{A,B\varrho \}$ is elliptic, 
$L_1$ is an elliptic selfadjoint mixed-order pseudodifferential
operator;  its domain is  $D(L_1)={\prod}_{k\in K_0
}H^{2m-k-\frac12}(\Sigma )$.

When $A_{B\varrho }$ is invertible, so are $T$, $L$ and $L_1$, and
\cite{G68} Th.\ II.1.4 implies
\begin{equation}
A_{B\varrho }^{-1}=A_\gamma ^{-1}+\inj_VT^{-1}\pr_V=A_\gamma ^{-1}+K_\gamma \Phi L_1^{-1}\Phi ^*K_\gamma ^*.\label{tag6.17}
\end{equation}
(It is used here that $\inj_V\gamma _V^{-1}\Phi =K_\gamma \Phi $.)

All this is just the implementation of the known results to operators
defined for the unbounded set $\Omega _+$. But now we are in a position to consider interesting
perturbations. 

We replace $T:V\to V$ for $A_{B\varrho }$ by an operator 
$\widetilde T:V\to V$, selfadjoint invertible  with a nonempty
essential spectrum, and want to see how
this effects the realization. As above, $\widetilde T$ carries over to 
\begin{equation}
\widetilde L_1=\Phi ^*(\gamma _V^*)^{-1}\widetilde T\gamma _V^{-1}\Phi :{\prod}_{k\in K_0
}H^{-k-\frac12}(\Sigma )\to {\prod}_{k\in K_0
}H^{k+\frac12}(\Sigma ),\label{tag6.18}
\end{equation}
with $D(\widetilde L_1)=\pr_1\gamma  D(\widetilde T)$, and the
boundary condition now takes the form
\begin{align}
\gamma u=\Phi \gamma _{K_0}u,&\quad \Phi ^* \chi u= \widetilde
G_1\gamma _{K_0}u,\quad \gamma _{K_0}u\in D(\widetilde L_1),\label{tag6.19}\\
 \text{ where
}
\widetilde G_1=& \widetilde
L_1+\Phi ^*P_{\gamma ,\chi }\Phi =G_1+\widetilde L_1-L_1.
\nonumber
\end{align}
Here
\begin{equation}
\wA^{-1}=A_\gamma ^{-1}+\inj_V\widetilde T^{-1}\pr_V=A_\gamma
^{-1}+K_\gamma \Phi \widetilde L_1^{-1}\Phi ^*K_\gamma ^*.\label{tag6.20}
\end{equation} 

\begin{theorem}\label{Theorem6.1} Consider the realization $A_{B\varrho }$
of $A$ in $L_2(\Omega _+)$ defined by a normal boundary condition
$B\varrho u=0$ (cf.\ {\rm \eqref{tag5.8}}) with $J\ne M_0$, and assume that
ellipticity and selfadjointness holds, cf.\ {\rm \eqref{tag6.7}--\eqref{tag6.9}}. 
$A_{B\varrho }$ corresponds to an operator $T:V\to V$, where
$V=K_\gamma X=K_\gamma \Phi {\prod}_{k\in K_0
}H^{-k-\frac12}(\Sigma )$, cf.\ also {\rm \eqref{tag6.12}, \eqref{tag6.13}, \eqref{tag6.15}}. 

Let $\widetilde T$ 
be a selfadjoint invertible operator in
$V$ with nonempty essential spectrum, and let $\widetilde A$ be the
realization of $A$ corresponding to $\widetilde T:V\to V$, i.e., where the boundary condition {\rm \eqref{tag6.7}},
equivalently written {\rm \eqref{tag6.11}}, is replaced by {\rm \eqref{tag6.19}}.
Then 
\begin{equation}
\sigma _{\operatorname{ess}}\wA=\sigma
_{\operatorname{ess}}A_0\cup\sigma _{\operatorname{ess}}\widetilde
T. \label{tag6.21} 
\end{equation}
\end{theorem}

\begin{proof} The proof goes in exactly the same way as the proof of
Theorem \ref{Theorem4.1} and Corollary \ref{Corollary4.2}. We cut up $\Omega _+$ in a bounded part
$\Omega _<$ and an exterior part $\Omega _>$, and use \eqref{tag6.20} and
Proposition \ref{Proposition5.2}
with $B\varrho =\gamma $ to see that $\widetilde A$ can be written as 
\begin{equation}
\wA^{-1}=(r^<K_\gamma \Phi  \widetilde L_1^{-1}\Phi ^*K_\gamma  ^*e^<)
\oplus (r^>A_{\gamma }^{-1}e^>)+S,\label{tag6.22}
\end{equation}
where $S$ is compact, the operator $r^>A_{\gamma }^{-1}e^>$ in $L_2(\Omega
_>)$ has the same essential
spectrum as $A_\gamma ^{-1}$, and the operator $r^<K_\gamma \Phi  \widetilde
L_1^{-1}\Phi ^*K_\gamma  ^*e^<$ in $L_2(\Omega _<)$ has the same essential 
spectrum as 
$K_\gamma \Phi  \widetilde L_1^{-1}\Phi ^*K_\gamma  ^*=\inj_V\widetilde
T^{-1}\pr_V$ outside 0. 
\end{proof}

Briefly expressed, the theorem states that any normal boundary
condition (apart from the Dirichlet condition) defining a selfadjoint
invertible 
elliptic realization, can be perturbed by addition of a suitable
operator to $G_1$ (the map from the free Dirichlet data to Neumann data)
to provide a selfadjoint invertible realization with a prescribed augmentation of
the essential spectrum.

\begin{example}\label{Example6.2}
Let $A=\Delta ^2+1$. Clearly, $A$ satisifes the positivity and
selfadjointness requirements, and it has the Green's formula \eqref{tag6.3}
with \begin{equation*}
\gamma =\{\gamma _0,\gamma _1\},\quad \chi =\{\chi _0,\chi _1\} =\{-\gamma  _1\Delta ,\gamma_0\Delta \} ,
\end{equation*}
as in \cite{BGW09} Example 3.14. The Dirichlet operator
\begin{equation}
{\cal A}_\gamma =\begin{pmatrix} \Delta ^2+1\\\gamma \end{pmatrix} : 
H^{s+4}(\Omega _+) \to
\begin{matrix}
H^{s}(\Omega _+)\\ \times \\ H^{s+\frac72}(\Sigma  )\times H^{s+\frac52}(\Sigma  )
\end{matrix},
\label{tag6.23}
\end{equation}
where $s>-\frac52$,
has an inverse $
\begin{pmatrix} R_\gamma &\; K_\gamma \end{pmatrix}$ continuous in the opposite
direction. Let us take (as in \cite{BGW09}, Ex.\ 3.14)
$J=\{0,2\}\subset M=\{0,1,2,3\}$; it satisifes \eqref{tag6.1}, and $J_0=\{0\}$,
$K_0=\{1\}$. With this choice, the boundary condition \eqref{tag6.7} is of the form
\begin{equation}
\gamma _0u=0,\quad      \gamma _0\Delta u=G_1\gamma _1u. \label{tag6.24}
\end{equation}
($F_0$ and $G_2$ vanish, being differential operators of
negative order.) $G_1$ is of order 1. Selfadjointness
of $A_{B\varrho }$ requires $G_1^*=G_1$, and if this holds and the
problem is elliptic, then $A_{B\varrho }$ is selfadjoint with domain
$D(A_{B\varrho })=\{u\in H^4(\Omega _+)\mid \text{\eqref{tag6.24} holds.}\}$
Continuing under this assumption, we find that
\begin{equation*}
X=\{0\}\times H^{-\frac32}(\Sigma ),\text{ naturally identified with }H^{-\frac32}(\Sigma ),
\end{equation*}
 and $L_1$ is the first-order pseudodifferential operator
\begin{equation}
L_1=G_1 -\pr_2 P_{\gamma ,\chi }\inj_2:H^{-\frac32}(\Sigma )\to
H^{\frac32}(\Sigma ),\label{tag6.25} 
\end{equation}
with $D(L_1)=H^{\frac52}(\Sigma )$ in view of the ellipticity. There
is a corresponding operator $T:V\to V$ where $V=K_\gamma (\{0\}\times
H^{-\frac32}(\Sigma )) $. Invertibility holds e.g.\ when $L_1$ has a
positive lower bound.

Replacing $T:V\to V$ by $\widetilde T:V\to V$, selfadjoint and
invertible with a nonempty essential
spectrum, corresponds to replacing $G_1$ by 
\begin{equation}
\widetilde G_1=G_1+\widetilde L_1-L_1, \quad
\widetilde L_1=\pr_2(\gamma _V^*)^{-1}\widetilde T\gamma
_V^{-1}\inj_2.\label{tag6.26 }
\end{equation}
The corresponding realization $\wA$ is defined by the boundary
condition \begin{equation}
\gamma _0u=0,\quad      \gamma _0\Delta u=\widetilde G_1\gamma _1u,
\quad \gamma _1u\in D(\widetilde L_1),
\label{tag6.27} 
\end{equation}
and
satisfies $\sigma _{\operatorname{ess}}\wA=\sigma
_{\operatorname{ess}}A_0\cup\sigma _{\operatorname{ess}}\widetilde T$.

\end{example}

\end{document}